\documentclass[11pt]{amsart}

\usepackage{amsmath,amsfonts,amssymb,amsthm}
\usepackage{txfonts}
\usepackage{latexsym,bm,graphicx}
\usepackage{mathrsfs}
\usepackage{color}
\usepackage{hyperref}

\title[Sharp gradient estimate for heat kernel...]{Sharp gradient estimate for heat kernels on $RCD^*(K,N)$ metric measure spaces}
\author{Jia-Cheng Huang}
\address{Jia-Cheng Huang: School of Mathematical Sciences, Fudan University,  Shanghai \\} \email{hjiach@mail2.sysu.edu.cn}
\author{Hui-Chun Zhang}
\address{Hui-Chun Zhang: Department of Mathematics, Sun Yat-sen University,  Guangzhou \\}
\email{zhanghc3@mail.sysu.edu.cn}

\newtheorem{thm}{Theorem}[section]
\newtheorem{thm1}{Theorem}

\newtheorem{lem}[thm]{Lemma}

\newtheorem{coro}[thm]{Corollary}

\theoremstyle{definition}
\theoremstyle{remark}

\newtheorem{defn}[thm]{Definition}
\newtheorem{rem}[thm]{Remark}

\numberwithin{equation}{section}

\newcommand{\ls}{\leqslant}
\newcommand{\gs}{\geqslant}

\newcommand{\ip}[2]{\langle{#1},{#2}\rangle}

\newcommand{\R}{\mathbb{R}}
\newcommand{\La}{\mathscr{L}}

\begin{document}

\begin{abstract}
In this paper, we will establish  an elliptic local Li-Yau   gradient estimate for weak solutions of the heat equation  on metric measure spaces with generalized Ricci curvature bounded from below. One of its main applications is a sharp gradient estimate for the logarithm of heat kernels.   These results seems new even for smooth Riemannian manifolds.
\end{abstract}

\maketitle

\section{Introduction}
The Li-Yau  gradient estimate for evolution equations is certainly central in   geometric analysis on Riemannian manifolds. One of the fundamental results is the following gradient estimates for  heat equations.
\begin{thm1}[Li-Yau \cite{ly86}]\label{thm-ly}
Let $M^n$ be an $n$-dimensional complete non-compact Riemannian manifold with $Ric(M^n)\gs-k$,  $k\gs0$. Let $B_{R}$ be a geodesic ball with radius $R$. If $u$ is a smooth positive solution of the heat equation $\Delta u=\partial_tu$ on $B_R\times(0,T)$, $0<T\ls \infty$, then
  \begin{equation}\label{eq1.1}
 \sup_{x\in B_{R/2}}\Big( |\nabla f|^2  -\alpha\cdot  \partial_t f \Big)(x,t) \ls \frac{C_n\cdot\alpha^2}{R^2}\Big(\frac{\alpha^2}{\alpha^2-1}+\sqrt kR\Big)+\frac{n\alpha^2k}{2(\alpha-1)}+\frac{n\alpha^2}{2t}
 \end{equation}
where $\alpha>1$, $f:=\ln u$, and $C_n$ is a constant depending only on $n$.
\end{thm1}
In another direction, Hamilton established an elliptic gradient estimate for bounded solutions of the heat equation on compact manifolds.
\begin{thm1}[Harmilton \cite{ham93}]\label{thm-hami}
Let  $M^n$ be an $n$-dimensional compact Riemannian manifold without boundary and with $Ric(M^n)\gs-k$,  $k\gs0$. Let $u$ be a smooth positive solution of the heat equation on $M^n\times(0,\infty)$. Suppose that $u\ls M$ on $M^n\times(0,\infty)$. Then, by setting $f:=\ln u$, we have
\begin{equation}\label{eq1.2}
|\nabla f|^2\ls \Big(\frac1 t +2k \Big)\ln \frac M u.
\end{equation}
\end{thm1}

In \cite{sz06}, Souplet and Zhang proved a localized elliptic gradient for heat equation on noncompact manifolds.
\begin{thm1}[Souplet-Zhang \cite{sz06}]\label{thm-sz}
 Let $(M^n,g)$ be an $n$-dimensional complete Riemannian manifold with $Ric(M^n)\gs-k$ with $k\gs0$.
Assume that $u$ is a smooth positive solution of the heat equation  on $B_{R}\times(0,T)$.
Suppose also that $u\ls M$ on $B_{R}\times(0,T)$. Then we have
\begin{equation}\label{eq1.3}
 |\nabla f|^2 \ls C_n\cdot\Big(\frac{1}{R^2}+\frac{1}{T}+k\Big) \Big( 1+\ln\frac{M}{u} \Big)^2, \qquad (x,t)\in  B_{R/2}\times(T/2,T),
\end{equation}
where $f:=\ln u$ and the constant $C_n$ depends only on the dimension $n$.
\end{thm1}

There is a rich literature on extensions and improvements of these Li-Yau's gradient estimates. Here, we refer some recent nice works and surveys on this topic, \cite{lee12,qia14,bbg16,lx11,bhllmy15,lixd12,nil08,lixd05} and so on.

In the pioneering works of Strum \cite{stu06-1,stu06-2} and Lott-Villani \cite{lv09,lv07-jfa}, an synthetic notion of lower Ricci bounds on metric spaces has been introduced. Up to now, many  improvements were given along this direction (see, for example, \cite{ags14,ags15,ams16,ags-duke,bs10,eks15,gig13} and so on). In particular, a satisfactory notion,
so-called \emph{Riemannian curvature-dimension condition} (denote by $RCD^*(K,N)$), was given in \cite{eks15,ams16}.  The constants $K$ and $N$   play the role of  ``Ricci curvature $\gs K$ and dimension $\ls N$".  Let $(X,d,\mu)$ be a metric measure space (a metric space equipped  a Radon measure)  satisfying $RCD^*(K,N)$ is a generalized notion for ``an Riemannian manifold with $Ricci\gs K$ and $\dim\ls N$".

Let $(X,d,\mu)$ be a metric measure space satifying $RCD^*(K,N)$, for some $K\in \mathbb R$ and $N\in[1,\infty)$.  Given any domain $\Omega\subset X$, according to \cite{che99,shan00}, the Sobolev spaces $W^{1,p}(\Omega)$, $1\ls p\ls\infty$, are well defined. Moreover,   the  space $W^{1,2}(\Omega)$ is a Hilbert space \cite{ags14}. Then, the weak solutions of  the heat equation on $\Omega$ are well defined. That is,
given an interval $I\subset\mathbb R$, a function $u(x,t)\in W_{\rm loc}^{1,2}(\Omega\times I)$ is called  a {\it (locally) weak solution
 for the heat equation on $\Omega\times I$} if it satisfies
\begin{equation*}
-\int_I\int_{\Omega}\!\ip{\nabla u}{\nabla \phi}d\mu dt = \int_I\int_{\Omega}\!\frac{\partial u}{\partial t}\cdot\phi d\mu dt
\end{equation*}
for all Lipschitz functions $\phi$ with compact support in $\Omega\times I$, where $\int_\Omega\ip{\nabla u}{\nabla \phi}d\mu$ is the inner product of $W^{1,2}(\Omega)$.  The local boundedness and the Harnack inequality for any such locally weak solutions of the heat equation have been established in \cite{stu95,stu96,mm13}.

In the case when $\Omega=X$ and $I=[0,\infty)$, the heat flow $(H_tf)_{t\gs0}$ with initial data $f\in L^2(X)$ provides a globally weak solution of the heat equation on $X$. By an abstract $\Gamma_2$-calculus for $(H_tf)$, some global versions of Li-Yau type gradient estimates  for $(H_tf)$ have been obtained (see \cite{qzz13,gm14,jia15,jiang-z16}). However, the locally  weak solutions $u(x,t)$  do not form a semi-group in general, and hence the method of $\Gamma_2$-calculus does not work in general.

  Our main result in this paper is the following  local gradient estimate  on $RCD^*(K,N)$ metric measure spaces . This is new even for smooth Riemannian manifolds!

\begin{thm}\label{thm1.1}
  Given $K\gs 0$ and $N\in(1,\infty)$, let $(X,d,\mu)$ be a metric measure space satisfying  $RCD^*(-K,N)$.
Let $T\in(0,\infty)$ and let  $B_{R}\subset X$ be a geodesic ball of radius $R$. Assume that  $u(x,t)\in W_{\rm loc}^{1,2}\big(B_{R,T}\big)$ is a  locally weak solution of the heat equation on $B_{R,T}:=B_{R}\times (0,T)$. Suppose also that there exist $M, m>0$ such that $m\ls u\ls M$ on $B_{R,T}$.   Then we have the  local gradient estimate:
\begin{equation}\label{eq1.4}
  |\nabla f|^2  (x,t)\ls C_N\cdot\Big(\frac{\ln(M/m)}{R^2}+\frac{1}{T}+K\Big)\cdot\ln\frac{ M}{ u(x,t)}
\end{equation}
for almost every $(x,t)\in  B_{R/2}\times(T/2,T)$, where $f=\ln u $ and the constant $C_N$  depends only on $N$.
\end{thm}
The second author and Zhu in \cite{zz-cvpde} have extended (\ref{eq1.1}) to general $RCD^*(K,N)$ metric measure spaces. Comparing with the argument in \cite{zz-cvpde},  there is a new technical difficulty in the proof of Theorem \ref{thm1.1}. Indeed,
recalling the proof of the elliptic Li-Yau gradient estimates (\ref{eq1.2}) and (\ref{eq1.3}) in the smooth case, we need a simple algebra inequality: for any $C^2$-functions $f,\phi$ on $(M^n,g)$, it holds
\begin{equation}\label{eq1.5}
\big|  f_{ij}+ \phi\cdot g_{ij}\big|^2\gs \frac{1}{n}\Big[trace\Big(  f_{ij}+ \phi\cdot g_{ij}\Big)\Big]^2= \frac{1}{n}\Big(\Delta f+n\phi  \Big)^2.
\end{equation}
We have \emph{not} an appropriate analogous of (\ref{eq1.5}) on general $RCD^*(K,N)$ metric measure spaces.  In this paper,  we find that this lack of (\ref{eq1.5}) can be compensated by an improvement of Bochner inequality, see also Remark \ref{rem2.7}.

One of the main application of Theorem \ref{thm1.1} is the following sharp gradient estimate for the logarithm of the heat kernel.
\begin{thm}\label{thm1.2}
  Given   $N\in(1,\infty)$ and $K\gs0$, let $(X,d,\mu)$ be a metric measure space satisfying  $RCD^*(-K,N)$.  Let  $H(x,y,t)$ be the heat kernel on $X$. Then there exists a  constant
  $C_{N,K} $, depending  only on $N$ and $K$,  such that, for almost every $(x,y,t)\in X\times X\times(0,\infty)$, we have
  \begin{equation*}
  |\nabla \ln H(x,y,t)|^2  \ls C_{N,K} \Big(\frac{ 1}{t}+K\Big)\cdot \Big(1+\frac{d^2(x,y)}{t}+ t\Big).
\end{equation*}
Moreover, in the case where $K=0$,   there exists a constant $C_N$, depending only on $N$, such that
 \begin{equation*}
  |\nabla \ln H(x,y,t)|^2 \ls\frac{C_N}{t}\cdot \Big(1+\frac{d^2(x,y)}{t}\Big).
\end{equation*}
\end{thm}

We give the result in the case of  smooth manifolds, which may be of independent interest.
\begin{coro}\label{cor-rimannian}
Let $(M^n,g)$ be an $n$-dimensional complete Riemannian manifold with $Ric(M^n)\gs -k$, $k\gs0$.  Let  $H(x,y,t)$ be the heat kernel on $M^n$. Then, we have
  \begin{equation*}
  |\nabla \ln H(x,y,t)|^2  \ls c_{n,k} \Big(\frac{ 1}{t}+k\Big)\cdot \Big(1+\frac{d^2(x,y)}{t}+ t\Big).
\end{equation*}
for  every $(x,y,t)\in M^n\times M^n\times(0,\infty)$, where $c_{n,k} $ is  a  constant
 depends  only on $n$ and $k$.

Moreover, in the case where $Ric(M^n)\gs0$,  we have, for some  constant $c_n$, depending only on $n$,   that
 \begin{equation*}
  |\nabla \ln H(x,y,t)|^2   \ls\frac{c_n}{t}\cdot \Big(1+\frac{d^2(x,y)}{t}\Big).
\end{equation*}
\end{coro}

\begin{rem} (1) \ \ This result is sharp, since one has, on $\mathbb R^n$, that
$$  |\nabla \ln H(x,y,t)|^2 (x,t) =\frac 1 t\cdot\frac{|x-y|^2}{4t}.$$
(2) \ \ In the case of compact manifolds, this result has been proved in \cite{hus99,str98,eng06} via Malliavin's  calculus.  In the case of non-compact manifolds, this result improves  the previous estimates of Souplet and Zhang   in \cite{sz06}, by using their elliptic Li-Yau gradient estimate Theorem\ref{thm-sz}. Very recently, under to add an assumption that the time is bounded, a simalar result has been obtained on non-compact Riemannian manifolds with appropriate  Bakry-Emery conditions by  Li \cite{lixd16}.
\end{rem}

The second application of Theorem \ref{thm1.1}  is the Lipschitz regularity of   locally weak solutions of the heat equation on $RCD^*(K,N)$  metric measure spaces. Let $u$ be a locally weak solution   of the heat equation on  $B_{R,T}:=B_R\times(0,T)$.  Recalling that the local boundedness and the Harnack inequality for $u$ have been established in \cite{stu95,stu96,mm13}. In particular, $u(\cdot,\cdot)$ must be locally H\"older continuity in $B_{R,T}$. On the other hand, in the case where $u(x,t)=H_tf$ is a global heat flow on $X$, the Lipschitz continuity of $u(\cdot,t)$ in $B_R$, for any $t\in(0,T)$, comes from the Bakry-Emery condition, see \cite{ags15,ags-duke}.
Here, from Theorem \ref{thm1.1},  we have the following locally Lipschitz continuity for  $u$.
\begin{coro}\label{cor1.5}
  Let  $K,N,X$ and $B_{R,T}$ be as in the above Theorem \ref{thm1.1}.  Assume $u$ be a locally weak solution of the heat equation on $B_{R,T}$. Then, for any $t\in(0,T)$, the function $u(\cdot,t)$ is Lipschitz continuous on $B_{R/2}.$
\end{coro}

 At last, two immediate consequences of Theorem \ref{thm1.1} is the following Hamilton's gradient estimates on non-compact case and a Liouville's theorem for   ancient solutions of the heat equation.
 \begin{coro}\label{cor1.6}
Let $X,K,N$ be as the above Theorem \ref{thm1.1} and let $T\in(0,\infty)$. Assume that  $u(x,t)\in W_{\rm loc}^{1,2}\big(X\times(0,T)\big)$ is a weak solution of the heat equation on $X\times (0,T)$. Suppose also that there exist $M,m>0$ such that $m\ls u\ls M$ on $X\times (0,T)$.   Then we have
\begin{equation*}
  |\nabla f|^2  (x,t)\ls C_N\cdot\Big(\frac{1}{T}+K\Big)\cdot\ln\frac{ M}{u(x,t)}
\end{equation*}
for almost every $(x,t)\in X\times(T/2,T)$, where $f=\ln u $ and the constant $C_N$  depends only on $N$.
\end{coro}

\begin{rem}
This is an extension of Theorem \ref{thm-hami} to non-compact spaces. Indeed, in Theorem \ref{thm-hami}, since $u$ is bounded from above and positive,    the Harnack inequality implies that $u$ must be  bounded from below by a positive number. Very recently, Theorem \ref{thm-hami} has been extended to non-compact Riemannian manifolds with appropriate  Bakry-Emery conditions by Li \cite{lixd16}.
 \end{rem}
\begin{coro}\label{cor1.8}
  Let  $(X,d,\mu)$  be a metric measure space satisfying $RCD^*(0,N)$ for some $N\in[1,\infty)$. Assume that $u(x,t)$ is an ancient solution of the heat equation on $X\times(-\infty,0]$.   If
\begin{equation}\label{eq1.6}
\liminf_{R\to\infty}\frac{\sup_{B_R\times(- R^2,0)}|u|}{R} =0,
\end{equation}
  then $u$ is a constant.
 \end{coro}

 \subsection*{Organization of the paper}
In Section 2, we will provide some necessary
materials on $RCD^*(K,N)$ metric measure spaces.
In  Section 3, we will prove Theorem \ref{thm1.1},Theorem \ref{thm1.2} and the Corollary \ref{cor1.5}, \ref{cor1.6} and \ref{cor1.8}.

{\noindent\bf{Acknowledgements}.} H. C. Zhang is partially supported by NSFC 11521101.

\section{Preliminaries}

Let $(X,d)$ be a proper (i.e.,  closed balls of finite radius are compact) complete metric space and $\mu$ be a Radon measure on $X$ with ${\rm supp}(\mu)=X.$ Denote by $B_r(x)$ the open ball centered at $x$ with radius $r$. For any open subset $\Omega\subset X$ and any $p\in[1,\infty]$, we denote by $L^p(\Omega):=L^p(\Omega,\mu)$.

\subsection{Riemannian curvature-dimension conditions  \emph{RCD*}(\emph{K,N}) }$\ $

The  \emph{curature-dimension condition} on  $(X,d,\mu)$  has been introduced by Sturm \cite{stu06-1,stu06-2}
 and Lott-Villani  \cite{lv09}. Given two constants $K\in \R $ and $N\in[1,\infty]$, the curvature-dimension condition $CD(K,N)$ is
  a synthetic notion  for ``generalized Ricci curvature   $\gs K$ and dimension $\ls N$" on  $(X,d,\mu)$.
  Bacher-Sturm \cite{bs10} introduced the \emph{reduced} curvature-dimension condition $CD^*(K,N)$, and  Ambrosio et al. \cite{ags15} introduced
   the  \emph{Riemannian} curvature-dimension condition $RCD(K,\infty)$.
Very recently,  Erbar et al. \cite{eks15} and Ambrosio et al. \cite{ams16} introduced a dimensional
 version of \emph{Riemannian curvature-dimension condition} $RCD^*(K,N)$.
 In the case of Riemannian geometry, the notion  $RCD^*(K,N)$ coincides with the original  Ricci curvature $\gs K$ and dimension $\ls N$. In the setting of  Alexandrov geometry, it is implied by generalized (sectional) curvature bounded below in the sense of Alexandrov \cite{pet11,zz10}.

We denote by $\mathscr P_2(X,d)$  the $L^2$-Wasserstein space over $(X,d)$, i.e., the set of all Borel probability measures $\nu$ with
$$\int_Xd^2(x_0,x)d\nu(x)<\infty$$
for some (hence for all) $x_0\in X$. Given $\nu_1,\nu_2\in \mathscr P_2(X,d)$, their $L^2$-Wasserstein distance is defined by
$$ W_2^2(\nu_0,\nu_1):=\inf\int_{X\times X}d^2(x,y)dq(x,y)$$
where the infimum is taken over all couplings $q$ of $\nu_1$ and $\nu_2$, i.e., Borel probability measures $q$ on $X\times X$ with marginals $\nu_0$ and $\nu_1.$ Such a coupling $q$ realizes the $L^2$-Wasserstein distance is called an \emph{optimal coupling} of $\nu_0$ and $\nu_1.$

Given a measure $\nu\in \mathscr P_2(X,d)$, its relative entropy is defined by
$$ {\rm Ent}(\nu):=\int_X\rho\ln \rho d\mu,$$
if $\nu=\rho\cdot\mu$ is absolutely continuous w.r.t. $\mu$ and $(\rho\ln \rho)_+$ is integrable. Otherwise we set ${\rm Ent}(\nu)=+\infty.$ Let $\mathscr P^*_2(X,d,\mu)\subset  \mathscr P_2(X,d)$ be the subset of all measures $\nu$ such that ${\rm Ent}(\nu)<\infty$.

\begin{defn}[\cite{eks15}]
Given $K\in\mathbb R$ and $N\in[1,\infty)$. A metric measure space $(X,d,\mu)$ is called to satisfy the \emph{entropy curvature-dimension condition} $CD^e(K,N)$ if any only if for each pair $\nu_0,\nu_1 \in \mathscr P^*_2(X,d,\mu)$ there exist a constant speed geodesic $(\nu_t)_{0\ls t\ls1}$ in $ \mathscr P^*_2(X,d,\mu)$ connecting $\nu_0$ to $\nu_1$ such that for all $t\in[0,1]$:
\begin{equation*}
\begin{split}
U_N(\nu_t)\gs  \sigma^{(1-t)}_{K/N}\big(W_2(\nu_0,\nu_1)\big)\cdot U_N(\nu_0)+ \sigma^{(t)}_{K/N}\big(W_2(\nu_0,\nu_1)\big)\cdot U_N(\nu_1),
\end{split}
\end{equation*}
where $U_N(\nu):=\exp\big(-\frac{1}{N}{\rm Ent}(\nu)\big)$ and the function
\begin{equation*}
\sigma^{(t)}_k(\theta):=
\begin{cases}
\frac{\sin(\sqrt k\cdot t\theta)}{\sin(\sqrt k\cdot \theta)},& \quad 0<k\theta^2<\pi^2,\\
t, &\quad  k\theta^2=0,\\
\frac{\sinh(\sqrt{-k}\cdot t\theta)}{\sinh(\sqrt{- k}\cdot \theta)},& \quad k\theta^2<0,\\
\infty,&\quad  k\theta^2\gs\pi^2.
\end{cases}
\end{equation*}
\end{defn}

Given a function $f\in C(X)$, the \emph{pointwise Lipschitz constant} (\cite{che99}) of $f$ at $x$ is defined by
\begin{equation*}
{\rm Lip}f(x):=\limsup_{y\to x}\frac{|f(y)-f(x)|}{d(x,y)}=\limsup_{r\to0}\sup_{d(x,y)\ls r}\frac{|f(y)-f(x)|}{r},
\end{equation*}
where we put ${\rm Lip}f(x)=0$ if $x$ is isolated. Clearly, ${\rm Lip}f$ is a $\mu$-measurable function on $X.$
The \emph{Cheeger energy}, denoted by ${\rm Ch}:\ L^2(X)\to[0,\infty]$, is defined  (\cite{ags14}) by
$${\rm Ch}(f):=\inf\Big\{\liminf_{j\to\infty}\frac 1 2\int_X({\rm Lip}f_j )^2d\mu\Big\},$$
where the infimum is taken over all sequences of Lipschitz functions $(f_j)_{j\in\mathbb N}$ converging to $f$ in $L^2(X).$ In general, ${\rm Ch}$ is a convex and lower semi-continuous functional on $L^2(X)$.
\begin{defn}[\cite{ags14,eks15}]
A metric measure space $(X,d,\mu)$ is called \emph{infinitesimally Hilbertian} if the associated Cheeger energy is quadratic. Moreover,  $(X,d,\mu)$ is said to satisfy \emph{Riemannian curvature-dimension condition} $RCD^*(K,N)$, for  $K\in\mathbb R$ and $N\in[1,\infty)$,  if it is infinitesimally Hilbertian and satisfies the $CD^e(K,N)$ condition.
\end{defn}

Let $(X,d,\mu)$ be a  metric measure space  satisfying   $RCD^*(K,N)$. For each $f\in D({\rm Ch}):=\{f\in L^2(X)\ :\ {\rm Ch}(f)<\infty\}$, it has
$${\rm Ch}(f)=\frac{1}{2}\int_X|\nabla f|^2d\mu,$$
where $|\nabla f|$ is the so-called minimal relaxed gradient of $f$ (see \S4 in \cite{ags14}). It was proved in \cite[Lemma 4.3]{ags14} that Lipschitz functions are
 dense in $ D(\rm Ch)$ in the sense that, for each   $f\in D(\rm Ch)$, there exist  a sequence of Lipschitz functions $(f_j)_{j\in\mathbb N}$ such that
 $f_j\to f$ in $L^2(X)$ and $|\nabla (f_j- f)|\to 0$ in $L^2(X)$.
Since the Cheeger energy ${\rm Ch}$ is a quadratic form, by the polarization, the minimal relaxed gradients bring an inner product as following: given $f,g\in D(\rm Ch)$, it was proved in \cite{gig15} that the limit
$$\ip{\nabla f}{\nabla g}:=\lim_{\epsilon\to0}\frac{|\nabla (f+\epsilon\cdot g)|^2-|\nabla f|^2}{2\epsilon}$$
exists in $L^1(X).$ The inner product is bi-linear and satisfies Cauchy-Schwarz inequality, Chain rule and Leibniz
rule (see Gigli \cite{gig15}).

 \subsection{Sobolev spaces and the weak Laplacian} $\  $

Given $K\in\mathbb R$ and $N\in[1,\infty)$. Let $(X,d,\mu)$ be an $RCD^*(K,N)$ metric measure space.  Several different notions of Sobolev spaces on  $(X,d,\mu)$   have been established in  \cite{che99,shan00,ags13-lip,hk00,haj03}. They coincide each other on  $RCD^*(K,N)$ metric measure spaces (see, for example, \cite{ags13-lip}).

Let $\Omega\subset X$ be a domain.  We denote by $Lip_{\rm loc}(\Omega)$ the set of locally Lipschitz continuous functions on $\Omega$, and by $Lip(\Omega)$ (resp. $Lip_0(\Omega)$) the set of Lipschitz continuous functions on $\Omega$ (resp, with compact support in $\Omega$).

For any $1\ls p\ls +\infty$ and
   $f\in Lip_{\rm loc}(\Omega)$, its $W^{1,p}(\Omega)$-norm is defined by
$$\|f\|_{W^{1,p}(\Omega)}:=\|f\|_{L^{p}(\Omega)}+\|{\rm Lip}f\|_{L^{p}(\Omega)}.$$
 The Sobolev spaces $W^{1,p}(\Omega)$ is defined by the closure of the set
$$\big\{f\in Lip_{\rm loc}(\Omega):\ \|f\|_{W^{1,p}(\Omega)}<+\infty\big\}$$
under  the $W^{1,p}(\Omega)$-norm.
The space $W_0^{1,p}(\Omega)$ is defined by the closure of $Lip_0(\Omega)$ under  the $W^{1,p}(\Omega)$-norm.
We say that a function $f\in W^{1,p}_{\rm loc}(\Omega)$ if $f\in W^{1,p}(\Omega')$ for every open subset $\Omega'\subset\subset\Omega.$

Since $(X,d,\mu)$ is assumed  to be infinitesimally Hilbertian, it is well known that $D({\rm Ch})=W^{1,2}(X)$, see, for example, \cite[Lemma 2.5]{zz-cvpde}.
 Given any   function $f\in W^{1,2}(X)$, the  $W^{1,2}$-norm of $f$, $\|f\|_{W^{1,2}(X)}=\|f\|_{L^2(X)}+2{\rm Ch}(f).$

 Fix any   open set $\Omega\subset X$ and $p\in(1,\infty)$. According to \cite[\S 4.1]{gig15}, the space $W^{1,2}(\Omega)$ is still a Hilbert space, and  for any $f,g\in W^{1,2}_{\rm loc}(\Omega)$, the function $|\nabla f|^2$ and $\ip{\nabla f}{\nabla g}$ are well defined in $ L^1_{\rm loc}(\Omega)$.
In the sequel of this paper, we will always denote by  $H^1_0(\Omega):=W^{1,2}_0(\Omega)$, $H^1(\Omega):=W^{1,2}(\Omega) $ and $H^1_{\rm loc}(\Omega):=W_{\rm loc}^{1,2}(\Omega).$

\begin{defn}
For each $f\in H^1_{\mathrm{loc}}(\Omega)$, the distribution $\mathscr Lf$ is a functional defined on  $H^1_0(\Omega)\cap L^\infty(\Omega)$ by
$$\mathscr Lf(\phi):=-\int_\Omega\ip{\nabla f}{\nabla \phi}d\mu\qquad \forall\ \phi\in H^1_0(\Omega)\cap L^\infty(\Omega).$$
For any $g\in H^{1}(\Omega) \cap L^{\infty}(\Omega)$, the distribution $g\cdot\mathscr Lf$ is defined by
\begin{equation*}
g\cdot\mathscr Lf(\phi):= \mathscr Lf(g\phi)\qquad \forall\ \phi\in H^1_0(\Omega)\cap L^\infty(\Omega).
\end{equation*}
\end{defn}
By the linearity of inner product  $\ip{\nabla f}{\nabla g}$, this distributional Laplacian is linear.

If, given  $f\in H_{\rm loc}^{1}(\Omega)$, there exists a function $g\in L^1_{\rm loc}(\Omega)$ such that
\begin{equation*}
\mathscr Lf(\phi)\gs\int_\Omega g\cdot\phi d\mu\qquad \forall\ 0\ls \phi\in H^1_0(\Omega)\cap L^\infty(\Omega),
\end{equation*}
then we say that  ``$\mathscr Lf\gs g $ in the sense of distributions". In this case, $\mathscr L f$ is a signed Radon measure, and hence, we also denote ``$\mathscr Lf\gs g\cdot \mu $  in the sense of measures". It is similar when we replace ``$\gs$" by ``$=$" or by ``$\ls$".

$\mathscr L$ satisfies the following Chain rule and Leibniz rule  \cite{gig15}, see also \cite[Lemma 3.2]{zz-cvpde}.
\begin{lem}[\cite{gig15,zz-cvpde}]\label{lem2.4}
Given $K\in \mathbb R$ and $N\in[1,\infty)$. Let $\Omega$ be an open domain of an $RCD^*(K,N)$ metric measure space $(X,d,\mu)$.\\
{\rm(i)\ (Chain rule)}\ \ \  Let $f\in  H^{1}(\Omega)\cap L^\infty(\Omega)$ and $\eta\in C^2(\mathbb R)$. Then we have, in the sense of distributions,
\begin{equation*}
\La[\eta(f)]=\eta'(f)\cdot\La f+\eta''(f)\cdot|\nabla f|^2.
\end{equation*}
{\rm(ii)\ (Leibniz rule)}\ \ \ Let $f,g\in  H^{1}(\Omega)\cap L^\infty(\Omega)$. Then we have, in the sense of distributions,
\begin{equation*}
\La(f\cdot g)=f\cdot\La g+g\cdot\La f+2\ip{\nabla f}{\nabla g}.
\end{equation*}
\end{lem}

Let $\Omega\subset X$ be a domain. Given $T>0$, we denote by $\Omega_{T}:=\Omega\times(0,T].$
\begin{defn}
 A function $u(x,t)\in H^1(\Omega_T) \ (=W^{1,2}(\Omega_{T}))$ is called a   {locally weak  solution}  of  the heat equation on $\Omega_{T}$  if for any  $[t_1,t_2]\subset(0,T)$ and any geodesic ball $B_R\subset\subset \Omega$,  we have
\begin{equation*}
 \int^{t_2}_{t_1}\int_{B_R}\Big(\partial_t u\cdot\phi+\ip{\nabla u}{\nabla \phi}\Big)d\mu dt = 0
\end{equation*}
 for all  $\phi(x,t)\in  Lip_0(B_{R}\times(t_1,t_2)\big).$
Here and in the sequel, we denote always $\partial_tu:=\frac{\partial u}{\partial t}.$
\end{defn}
The local boundedness and the Harnack inequality for such weak solutions have been proved by Sturm \cite{stu95,stu96} in the setting of abstract local Dirichlet form and by Marola and Masson \cite{mm13} in the setting of metric measure space with a standard volume doubling property and supporting a $L^2$-Poincar\'e inequality. Of course,  it is vivid  for metric measure spaces satisfying $RCD^*(K,N)$ for some $K\in\mathbb R$ and $N\in[1,\infty).$

\subsection{The localized Bochner formula}$\ $

Combining with the works \cite{eks15, ams16} on a global version of Bochner formula  and  a good cut-off function in \cite{ams16,mn14,hkx13}, one can obtain the following localized  Bochner formula, see \cite[Corollary 3.6]{zz-cvpde} for details.
\begin{thm}[\cite{zz-cvpde}]\label{thm2.6}

 Let $(X,d,\mu)$  be a  metric measure space  satisfying   $RCD^*(K,N)$ for $K\in\mathbb R$ and $N\in[1,\infty)$. Let $B_R$ be a geodesic ball with radius $R$ and centered at a fixed point $x_0$.

  Assume that $f\in H^{1}(B_R)$ satisfies
$\mathscr Lf=g $ on $B_R$ in the sense of distributions with the function $g\in H^1(B_R)\cap L^\infty(B_R)$.
Then we have $|\nabla f|^2\in H^{1}(B_{R/2}) \cap L^\infty(B_{R/2})$  and that the distribution $\La(|\nabla f|^2)$ is a signed Radon measure on  $B_{R/2}$. If its Radon-Nikodym decomposition w.r.t. $\mu$ is denoted by
$$\La(|\nabla f|^2)=  \La^{\rm ac}(|\nabla f|^2)\cdot\mu+\La^{\rm sing}(|\nabla f|^2),$$
then we have  $\La^{\rm sing}(|\nabla f|^2)\gs0$ and, for $\mu$-a.e. $x\in B_{R/2},$
\begin{equation*}
 \frac{1}{2}\La^{\rm ac}(|\nabla  f |^2 ) \gs  \frac{g^2}{N}+\ip{\nabla f}{\nabla  g}+K|\nabla f|^2.
 \end{equation*}
Furthermore, if $N>1$, we have,  for  $\mu$-a.e. $  x\in B_{R/2}\cap \big\{y:\ |\nabla f(y)|\not=0\big\}$, an improvement estimate
\begin{equation}\label{eq2.1}
\frac{1}{2} \La^{\rm ac}(|\nabla f|^2)\gs \frac{ g^2}{N}+ \ip{\nabla f}{\nabla g}+K|\nabla f|^2+\frac{N}{N-1}\cdot\Big(\frac{\ip{ \nabla f}{\nabla |\nabla f|^2}}{2|\nabla f|^2}-\frac{ g}{N}\Big)^2.
\end{equation}
\end{thm}

\begin{rem}\label{rem2.7}
In this paper, the key fact is that the last term in the improvement estimate (\ref{eq2.1}) is enough to  compensate for    the lack of (\ref{eq1.5}) in general $RCD^*(K,N)$ metric measure spaces. In the smooth case, this idea is used in \cite{bq00}.
\end{rem}

\subsection{The pointwise maximum principles}$\ $

The pointwise maximum principle  states that, given a $C^2$-function $f$ defined on a smooth manifold $(M,g)$, if $f$ achieves one of its local maximum at point $x_0\in M$, then we have $\nabla f(x_0)=0$ and $\Delta f(x_0)\ls0.$  It is a powerful tool in geometric analysis.  However, it does not make sense on singular metric measure space in general.

To compensate for the lack of the pointwise maximum principle, the second author and Zhu in \cite[Theorem 1.3]{zz-cvpde} proved the following  analogous tool on general $RCD^{*}(K,N)$-spaces.
 \begin{thm}[\cite{zz-cvpde}]
Let $\Omega$ be a bounded domain in a metric measure space $(X,d,\mu)$ with $RCD^*(K,N)$ for some $K\in\mathbb R$ and $N\gs1.$
Let $f(x)\in H^{1}(\Omega)\cap L^\infty_{\rm loc}(\Omega)$  such that  $\La f $ is a signed Radon measure with $\La^{\rm sing}f\gs 0$, where $\La^{\rm sing}f$ is the singular part with respect to $\mu$.
Suppose that $f$ achieves one of its  strict maximum in  $\Omega$ in the sense that: there exists a neighborhood $U\subset\subset \Omega$ such that
\begin{equation*}
\sup_{U}f>\sup_{\Omega\backslash U}f.
\end{equation*}
Then, given any $w\in  H^{1}(\Omega)\cap L^\infty(\Omega)$, there exists a sequence of points $\{x_j\}_{j\in\mathbb N}\subset U$  such that they are the approximate continuity points of $\La^{\rm ac}f$ and $\ip{\nabla f}{\nabla w}$, and that
$$f(x_j)\gs \sup_\Omega f-1/j\qquad{\rm and }\qquad  \La^{\rm ac}f(x_j)+\ip{\nabla f}{\nabla w}(x_j)\ls 1/j.$$
\end{thm}

\begin{rem}
The assumption that singular part $\La^{\rm sing}f\gs 0$ is necessary. Let us consider a simple example: $f(t)=-|t|$ defined on $(-1,1)$. Then $f''(t)=-\delta(0)$, the  Dirac measure with center at $s=0$. By choosing $w=f$, then, at each the approximate continuity points of $f''$ and $f'\cdot w'=f'^2$, we have $f''+f'w'=f''+f'^2=1.$
\end{rem}

We will need also the following parabolic version (see \cite[Theorem 4.4]{zz-cvpde}):
\begin{thm}[\cite{zz-cvpde}]\label{max-p}
Let $\Omega$ be a bounded domain  and let $T>0$.
Let $f(x,t)\in H^{1}(\Omega_T)\cap L^\infty(\Omega_T)$ and suppose that $f$ achieves one of its strict maximum in $\Omega\times(0,T]$ in the sense that:  there exists a neighborhood $U\subset\subset \Omega$ and an interval $(\delta,T]\subset   (0,T]$ for some $\delta>0$ such that
$$\sup_{U\times(\delta,T]}f>\sup_{\Omega_T\backslash ( U\times(\delta,T])}f .$$
Assume that, for almost every $t\in(0,T)$,  $\La f(\cdot, t) $ is a signed Radon measure with $\La^{\rm sing}f(\cdot,t)\gs 0$.
Let  $w\in  H^{1}(\Omega_T)\cap L^\infty(\Omega_T)$  with $\partial_tw(x,t)\ls C$ for some constant $C>0$, for almost all $(x,t)\in\Omega_T$.  Then,  there exists a sequence of points $\{(x_j,t_j)\}_{j\in\mathbb N}\subset U\times(\delta,T]$  such that every $x_j$ is  an approximate continuity point of $\La^{\rm ac}f(\cdot, t_j)$ and $\ip{\nabla f}{\nabla w}(\cdot, t_j)$, and that
$$f(x_j,t_j)\gs \sup_{\Omega_T} f-1/j\quad{\rm and }\quad  \La^{\rm ac}f(x_j,t_j)+\ip{\nabla f}{\nabla w}(x_j,t_j)-\frac{\partial}{\partial t}f(x_j,t_j)\ls 1/j.$$
\end{thm}

\section{The proof of Theorem \ref{thm1.1} and its consequences}

Given $K\in\mathbb R$ and $N\in [1,\infty)$. Let $(X,d,\mu)$  be  a metric measure space satisfying $RCD^*(K,N)$.  In this section, we will prove   Theorem \ref{thm1.1}  and its consequences.

The following is the main lemma in this paper.
\begin{lem}\label{lem3.1}
 Given $K\gs0$ and $N\in(1,\infty)$. Let $(X,d,\mu)$ be a metric measure space satisfying $RCD^*(-K,N)$.
Let $f(x,t)$ be a function on $B_{R,T}$ with all  $f, \partial_t f, |\nabla f|^2\in H^{1}(B_{R,T})\cap L^\infty(B_{R,T})$.
Assume that, for almost all $t\in(0,T)$, the function $f(\cdot,t)$ satisfies
\begin{equation}\label{eq3.1}
\mathscr Lf= \partial_tf-|\nabla f|^2\quad {\rm on}\ \ B_R
\end{equation}
in the sense of distributions. Let $\eta\in C^2(\mathbb R)$ with $\eta(f)>0$.
We  put
 $$F:=\eta(f)\cdot|\nabla f|^2.$$
 Then, we have that $F\in H^{1} (B_{R,T} )\cap L^{\infty}(B_{R,T} )$
and that, for almost all $t\in(0,T)$, the function $F(\cdot,t)$ satisfies that $\mathscr LF$ is a signed Radon measure on $B_{R/2}$ with  $\La^{\rm sing} F\gs0$ and
\begin{equation}\label{eq3.2}
\La^{\rm ac} F-\partial_tF+2\ip{\nabla f}{\nabla F} \gs \frac{\eta\cdot (\eta''+\eta')-2(\eta')^2}{\eta^3}\cdot F^2-2KF,
\end{equation}
$\mu{\rm-a.e.}$  on $ B_{R/2}$. Here and in the sequel, we denote by $\eta:=\eta(f),$  $\eta':=\eta'(f) $ and  $\eta'':=\eta''(f).$
\end{lem}
\begin{proof}
From  the assumption $f,|\nabla f|^2\in H^{1}(B_{R,T})\cap L^\infty(B_{R,T})$, we know that
 $F\in H^{1} (B_{R,T} )\cap L^{\infty}(B_{R,T} )$. Set
 $$g(x,t):=\partial_tf-|\nabla f|^2 \in H^{1}(B_{R,T})\cap L^\infty(B_{R,T}).$$
The Fubini Theorem implies that, for almost all $t\in(0,T)$, the function $g(\cdot,t) \in H^{1}(B_{R})\cap L^\infty(B_{R}).$  Therefore, by using Theorem \ref{thm2.6} to  $RCD^*(-K,N)$-spaces, we
get that the distribution $\La(|\nabla f|^2)$ is a signed Radon measure on  $B_{R/2}$, which satisfies that   $\La^{\rm sing}(|\nabla f|^2)\gs0$ and, for $\mu$-a.e. $x\in B_{R/2},$ that
\begin{equation}\label{eq3.3}
\La^{\rm ac}(|\nabla  f |^2 ) \gs   \frac{2g^2}{N}+2\ip{\nabla f}{\nabla  g}-2K|\nabla f|^2,
 \end{equation}
and moreover,  for  $\mu$-a.e. $  x\in B_{R/2}\cap \big\{y:\ |\nabla f(y)|\not=0\big\}$, we have
\begin{equation}\label{eq3.4}
  \La^{\rm ac}(|\nabla f|^2)\gs  \frac{2 g^2}{N}+2 \ip{\nabla f}{\nabla g}-2K|\nabla f|^2+\frac{2N}{N-1}\cdot\Big(\frac{\ip{ \nabla f}{\nabla |\nabla f|^2}}{2|\nabla f|^2}-\frac{ g}{N}\Big)^2.
\end{equation}
From Lemma \ref{lem2.4}, we get, for almost all $t\in (0,T)$, that the function $F(\cdot,t)$ is a signed Radon measure on $B_{R/2}$ with
\begin{equation*}
\begin{split}
 \mathscr LF&=|\nabla f|^2\cdot \mathscr L(\eta )+\eta \cdot \mathscr L(|\nabla f|^2)+2\ip{\nabla |\nabla f|^2}{\nabla \eta }\cdot\mu\\
 &=\eta' \cdot   |\nabla f|^2\cdot \mathscr L(f)+\eta'' \cdot |\nabla f|^4\cdot\mu+\eta \cdot\mathscr L(|\nabla f|^2)+2\eta' \ip{\nabla |\nabla f|^2}{\nabla f}\cdot\mu
 \end{split}\end{equation*}
in the sense of measures.   This implies that
\begin{equation*}
 \mathscr L^{\rm sing}F=\eta' \cdot   |\nabla f|^2\cdot \mathscr L^{\rm sing}f+ \eta \cdot\mathscr L^{\rm sing}(|\nabla f|^2) \gs0,
\end{equation*}
since $ \eta(f)>0$ and $\mathscr L^{\rm sing}f=0$, and that
\begin{equation}\label{eq3.5}
\begin{split}
 \mathscr L^{\rm ac}F=\eta' \cdot   |\nabla f|^2\cdot g+\eta'' \cdot |\nabla f|^4 +\eta \cdot\mathscr L^{\rm ac}(|\nabla f|^2)+2\eta'\cdot \ip{\nabla |\nabla f|^2}{\nabla f}
 \end{split}
 \end{equation}
for almost all $x\in B_{R/2}$. From $F\in  H^{1}(B_{R,T})$, we get
\begin{equation}\label{eq3.6}
  \partial_tF=\eta' \cdot\partial_tf\cdot   |\nabla f|^2 +2\eta  \cdot   \ip{\nabla\partial_tf}{\nabla f}
 \end{equation}
for almost all $x\in B_{R}$.

\noindent\emph{Case 1}: Let us consider points in $B_{R/2}\cap \big\{y:\ |\nabla f(y)|\not=0\big\}$.
 Denote by, at points where $|\nabla f|\not=0$,
$$A:=|\nabla f|^2\quad {\rm and}\quad B:=\frac{\ip{\nabla f}{\nabla |\nabla f|^2}}{|\nabla f|^2}.$$
By combining  (\ref{eq3.4}),  (\ref{eq3.5}),  (\ref{eq3.6}) and $g=\partial_tf-A$, we  concludes that
\begin{equation*}
\begin{split}
 \mathscr L^{\rm ac}F-\partial_tF&=\eta' \cdot   A\cdot(g-\partial_tf)+\eta'' \cdot A^2 +\eta \cdot\mathscr L^{\rm ac}(|\nabla f|^2)+2\eta'\cdot AB-2\eta\ip{\nabla f}{\nabla(g+A)}\\
 &\gs (-\eta'+\eta'')\cdot A^2+2\eta'\cdot AB-2\eta\ip{\nabla f}{\nabla(g+A)}\\
 &\quad +2\eta\Big(\frac{g^2}{N}+\ip{\nabla f}{\nabla g}-KA+\frac{N}{N-1}\big(\frac{B}{2}-\frac{g}{N}\big)^2\Big)\\
 &= (-\eta'+\eta'')\cdot A^2+2\eta'\cdot AB-2\eta\ip{\nabla f}{\nabla (F/\eta) }\\
 &\quad +2\eta\Big(\frac{g^2}{N-1} -KA+ \frac{N\cdot B^2}{4(N-1)}-\frac{Bg}{N-1} \Big)\\
 &\gs (-\eta'+\eta'')\cdot A^2+2 \eta' \cdot AB-2 \ip{\nabla f}{\nabla F}+2\eta'\cdot A^2 -2\eta KA
 +  \frac{\eta\cdot B^2}{2} \\
 \end{split}
 \end{equation*}
 for  almost all $  x\in B_{R/2}\cap \big\{y:\ |\nabla f(y)|\not=0\big\}$, where we have used
 $$F=\eta A\quad {\rm and}\quad g^2-Bg+\frac{N\cdot B^2}{4}\gs \frac{(N-1)\cdot B^2}{4}.$$
Noticing that $\eta>0$ and
$$\frac{\eta\cdot B^2}{2}+ 2 \eta' \cdot AB\gs -\frac{2(\eta' )^2}{\eta}\cdot A^2,$$
 we have,    $\mu$-a.e. on  $    B_{R/2}\cap \big\{y:\ |\nabla f(y)|\not=0\big\}$,  that (since  $F= \eta A$,)
 \begin{equation}\label{eq3.7}
\begin{split}
 \mathscr L^{\rm ac}F-\partial_tF& \gs (\eta'+\eta'')\cdot A^2
 -\frac{2(\eta')^2}{\eta}\cdot A^2-2\ip{\nabla f}{\nabla F} -2\eta KA\\
 &=\frac{\eta\cdot (\eta''+\eta')-2(\eta')^2}{\eta^3}\cdot F^2-2\ip{\nabla f}{\nabla F}-2KF.
 \end{split}
 \end{equation}

\noindent \emph{Case 2}: Let us consider points in $B_{R/2}\cap \big\{y:\ |\nabla f(y)|=0\big\}$.  At point where $A=|\nabla f|^2=0$, the combination  of (\ref{eq3.3}),  (\ref{eq3.5}) and  (\ref{eq3.6}) implies that
\begin{equation*}
\begin{split}
 \mathscr L^{\rm ac}F-\partial_tF&= \eta \cdot\mathscr L^{\rm ac}(|\nabla f|^2)+2\eta'\cdot \ip{\nabla |\nabla f|^2}{\nabla f}-2\eta\ip{\nabla f}{\nabla(g+A)}\\
 &\gs2\eta'\cdot \ip{\nabla A}{\nabla f}-2\eta\ip{\nabla f}{\nabla(g+A)} +2\eta\Big(\frac{g^2}{N}+\ip{\nabla f}{\nabla g}-KA\Big)\\
 &=2(\eta'-\eta)\ip{\nabla f}{\nabla A } +\frac{2\eta \cdot g^2}{N}\gs2(\eta'-\eta)\ip{\nabla f}{\nabla A}.
 \end{split}
 \end{equation*}
Noticing that, $\mu-{\rm a.e.}$ on $\big\{y:\ |\nabla f|(y)=0\big\}$, and that $ |\ip{\nabla f}{\nabla A}|\ls |\nabla f|\cdot|\nabla A|$, we conclude that $|\ip{\nabla f}{\nabla A}|=0$ for almost all $x\in \big\{y:\ |\nabla f|(y)=0\big\}$. The same holds for  $|\ip{\nabla f}{\nabla F}|$.
\begin{equation}\label{eq3.8}
\begin{split}
 \mathscr L^{\rm ac}F-\partial_tF \gs0=-2\ip{\nabla f}{\nabla F}, \quad \mu{\rm -a.e.}\ \ {\rm on}\ \ B_{R/2}\cap \big\{y:\ |\nabla f(y)|=0\big\}.
 \end{split}
 \end{equation}

The combination of (\ref{eq3.7}) in Case 1 and (\ref{eq3.8}) in Case 2 gives the result (\ref{eq3.2}).
\end{proof}

We firstly show the main result under an added assumption that  $u(x,t)$ satisfies
\begin{equation}\label{eq3.9}
u\in H^{1}(B_{R,T})\cap L^\infty(B_{R,T})\qquad {\rm and}\qquad \partial_t u \in H^{1}(B_{R,T})\cap L^\infty(B_{R,T}).
\end{equation}
\begin{lem}\label{lem3.2}
Let $K\gs 0$ and $N\in(1,\infty)$, and let $(X,d,\mu)$ be a metric measure space satisfying  $RCD^*(-K,N)$.
Let $T\in(0,\infty)$ and let  $B_{R}$ be a geodesic ball of radius $ R$, $B_{R,T}=B_{R}\times (0,T)$, and let  $u(x,t)\in H^{1}\big(B_{R,T}\big)\cap L^\infty(B_{R,T})$ be a positive locally weak solution of the heat equation on $B_{R,T}$. Suppose that there exists $\delta,\delta'\in(0,1/2)$ such that $\delta\ls u\ls 1-\delta'$. Suppose also $\partial_t u \in H^1(B_{R,T})\cap L^\infty(B_{R,T})$. Then,   we have the following local gradient estimate
\begin{equation}\label{eq3.10}
 \sup_{B_{R/2}\times{(T/2,T)}} \frac{|\nabla f|^2 }{-f}(x,t) \ls C_N \cdot \Big( \frac{\ln(1/\delta)}{R^2}+\frac{1}{T}+K\Big),
\end{equation}
where $f=\ln u$.
\end{lem}

\begin{proof} From the assumption of $u$, we know $f,\partial_tf=\partial_t u/u\in H^1(B_{R,T})\cap L^\infty(B_{R,T})$. It was proved in \cite[Lemma 5.5]{zz-cvpde} that $|\nabla u|^2\in H^1(B_{3R/4,T})\cap L^\infty(B_{3R/4,T})$. Hence, by  $1-\delta'\gs u\gs \delta>0$, we have
$|\nabla f|^2= |\nabla u|^2/u^2 \in H^1(B_{3R/4,T})\cap L^\infty(B_{3R/4,T}).$ By Fubini Theorem, it is clear that, for almost all $t\in(0,T)$, the function $u(\cdot, t)$ satisfies $\mathscr Lu=\partial_t u$ on $B_{R}$ in the sense of distributions. From Lemma \ref{lem2.4}, for almost all $t\in(0,T)$, the function $f(\cdot,t)$ satisfies
\eqref{eq3.1}. Therefore, $f$ satisfies all of assumptions in Lemma \ref{lem3.1}.

Let $\eta(s)=-1/s$ and consider $F(x,t):=\eta(f)\cdot |\nabla f|^2$ on $B_{3R/4,T}$. Since $\ln\delta\ls f\ls \ln(1-\delta')<0$, we have $\eta(f)\gs \frac{1}{\ln(1/\delta)}>0$.
According to Lemma \ref{lem3.1} and
$$\eta(\eta''+\eta')-2(\eta')^2=\eta^3,$$
 we conclude that  $F\in H^{1} (B_{3R/4,T} )\cap L^{\infty}(B_{3R/4,T} )$
and that, for almost all $t\in(0,T)$, the function $F(\cdot,t)$ satisfies that $\mathscr LF$ is a signed Radon measure on $B_{5R/8}$ with  $\La^{\rm sing} F\gs0$ and
\begin{equation}\label{eq3.11}
\La^{\rm ac} F-\partial_tF+2\ip{\nabla f}{\nabla F} \gs   F^2-2KF,\quad \mu{\rm -a.e.}\ \ {\rm on}\ \ B_{5R/8}.
\end{equation}

We can assume that $ \sup_{B_{R/2}\times{(T/2,T)}}F>0$, otherwise, we are done. Fix a constant $\epsilon>0$ to be sufficiently small, (for example, we can choose $\epsilon:=\min\{1/10,\sup_{B_{R/2}\times{(T/2,T)}}F/10\}).$ We now choose a modified cutoff in the spatial direction $\phi(x)=\phi(r(x))$ to be a function of the distance $r$ to the center of  $B_R$ with the following properties that
$$  \epsilon\ls \phi \ls1\ \ {\rm on}\ \ B_{R},\quad \phi =1\ \ {\rm on}\ \ B_{R/2},\quad \phi = \epsilon\ \ {\rm on}\ \ B_{R}\backslash B_{9R/16},$$
and
$$ - \frac{C}{R}\phi ^{\frac 1 2}\ls  \phi '(r)  \ls 0\quad {\rm and}\quad  |\phi ''(r)|   \ls \frac{C}{R^2} \qquad \forall\ r\in(0,R)$$
for some  universal constant $C$ (which is independent of  $N,K,R$). Then, according to the Laplacian comparison theorem  \cite[Corollary 5.15]{gig15} on $RCD^*(-K,N)$ spaces, we have that
\begin{equation}
\frac{|\nabla \phi |^2}{\phi }\ls \frac{C_1}{R^2}\qquad {\rm and}\qquad \La\phi \gs-C_2(\frac{\sqrt K}{ R}+\frac 1{R^2})
\end{equation}
on $B_{R}$. (A detailed calculation for this can be found on Page 22 in \cite{zz-cvpde}.) Here and in the sequel of this proof, we denote $C_{1},C_2,C_3,\cdots$  the various constants which  depend only on  $N$.  Hence, the distribution $\La \phi$ is a signed Radon measure with $(\La \phi)^{\rm ac}\gs -C_2(\sqrt K/R+1/R^2)$ a.e. $x\in B_{R},$ and  $(\La \phi)^{\rm sing}\gs0.$ Let $\xi(t)$ be a function on $(0,T)$ such that
\begin{equation}\label{eq3.13}
\xi(t)=\epsilon \ \ {\rm on} \ \ (0,T/4),\quad  \xi(t)=1 \ \ {\rm on}\ \ (T/2,T) \quad\ {\rm and}\quad   0\ls \xi'(t)\ls \frac{C_3}{T} \ \ {\rm on}\ \ (0,T).
\end{equation}
Let $\psi(x,t):=\phi(x)\xi(t)$ defined on $B_{R,T}$.

Put $G(x,t):= \psi F$. Then we have $G\in    H^{1} (B_{3R/4,T} )\cap L^{\infty}(B_{3R/4,T} )$.  According to Lemma   \ref{lem2.4}, we have, for almost every $t\in(0,T)$,   that the function $G(\cdot, t)$ satisfies that
$$\La G=F\La \psi+\psi\La F+2\ip{\nabla \psi}{\nabla F}$$
in the sense of distributions.  Fix arbitrarily a such $t\in(0,T)$. Then $\La G$ is a signed Radon measure on $B_{5R/8}$ with
\begin{equation}\label{eq3.14}
(\La G)^{\rm sing}=F(\La \psi)^{\rm sing}+\psi(\La F)^{\rm sing}\gs0,
\end{equation}
since $F\gs0$, and that
$(\La G)^{\rm ac}=F(\La \psi)^{\rm ac}+\psi(\La F)^{\rm ac}+2\ip{\nabla \psi}{\nabla F}$ a.e. $x\in B_{5R/8}$. We have, for almost all $x\in B_{5R/8}$,
\begin{equation} \label{eq3.15}
\begin{split}
(\La G)^{\rm ac}-\partial_t G  +2\ip{\nabla f}{\nabla G} =&  \psi\Big((\La F)^{\rm ac} -\partial_t F  +2\ip{\nabla f}{\nabla F}\Big)\\
&\ +  F(\La \psi)^{\rm ac}+2\ip{\nabla \psi}{\nabla F}  +2\ip{\nabla f}{\nabla\psi}F-F\cdot\partial_t\psi.
\end{split}
\end{equation}
By (\ref{eq3.11})-(\ref{eq3.13}) and $G= \psi F$,  we have, for almost all $x\in  B_{5R/8}$,   that
  \begin{equation}\label{eq3.16}
\begin{split}
{\rm RHS \ of}\  (\ref{eq3.15})  \gs &  \psi\Big[F^2-2KF\Big]  +G\frac{(\La \psi)^{\rm ac}}{\psi}+2\ip{\nabla \psi}{\nabla (G/\psi)}  +2\ip{\nabla f}{\nabla\psi}\frac{G}{\psi}- \frac{G}{\psi}\cdot\partial_t\psi \\
\gs\ & \frac{G^2}{\psi}-2KG +\frac{G}{\psi}\Big[-C_2\big(\frac{\sqrt K}{R}+\frac{1}{R^2}\big)-\frac{2C_1}{R^2}-\frac{C_3}{T}\Big] \\
&\ +2 \ip{\nabla \psi}{\nabla  G}/\psi  -2|\nabla f|\cdot|\nabla\psi|\frac{G}{\psi}.
\end{split}
\end{equation}
Notice that $G=\psi\eta\cdot|\nabla f|^2$, $\eta(f)\gs\frac{1}{-\ln\delta}$,  and that
$$2|\nabla f|\cdot|\nabla\psi|=2\sqrt{\psi\eta}|\nabla f|\cdot\frac{|\nabla\psi|}{\sqrt{\psi\eta}}\ls \frac{G}{2}+2\frac{|\nabla\psi|^2}{\psi\eta}\ls \frac{G}{2}+\frac{2 C_1}{R^2}\cdot\ln(1/\delta).$$
Therefore, substituted this in (\ref{eq3.16}), we have, for almost all $x\in  B_{5R/8}$,
  \begin{equation}\label{eq3.17}
\begin{split}
{\rm RHS \ of}\  (\ref{eq3.15})
\gs\ & \frac{G^2}{2\psi}-2KG +\frac{G}{\psi}\Big[-C_2\big(\frac{\sqrt K}{R}+\frac{1}{R^2}\big)-\frac{2C_1}{R^2}-\frac{C_3}{T}-\frac{2 C_1}{R^2}\cdot\ln(1/\delta)\Big] \\
&\ +2 \ip{\nabla \psi}{\nabla  G}/\psi.
\end{split}
\end{equation}
By combining \eqref{eq3.17} with (\ref{eq3.15}) and $K\gs0$, $0<\psi\ls1$, we get, for almost all $x\in  B_{5R/8}$,
  \begin{equation}\label{eq3.18}
\begin{split}
(\La G)^{\rm ac}-\partial_t G  +2\ip{\nabla f}{\nabla G}-2 \ip{\nabla \ln\psi}{\nabla  G}
\gs\ & \frac{G^2}{2\psi}- \frac{G}{\psi}\Big[2K+A(K,R,T,\delta)\Big]
 \end{split}
\end{equation}
where
$$A(K,R,T,\delta):= C_4\Big(\frac{\sqrt K}{R}+\frac{ 1+\ln(1/\delta)}{R^2} +\frac{1}{T} \Big), \qquad {\rm and}\qquad C_4:=\max\{C_2+2C_1,C_3\}.$$
It is clear that $G$ achieves its strict maximum in $B_{5R/8}\times(T/4,T]$ in the sense of Theorem \ref{max-p}. Notice that $\partial_tf\in L^\infty(B_{R,T})$ and (\ref{eq3.14}), we can use Theorem \ref{max-p} to $G$, with choosing $w:=2f-2\ln\psi\in H^1(B_{5R/8,T})\cap L^\infty(B_{5R/8,T})$, and combining with \eqref{eq3.18}, to conclude that there exist a sequence $(x_j,t_j)_{j\in\mathbb N}$ such that, for each $j\in \mathbb N$,
$$G(x_j,t_j)\gs \sup_{B_{5R/8,T}}G-1/j$$
 and that
$$ \frac{G^2}{2\psi}(x_j,t_j)- \frac{G}{\psi}(x_j,t_j)\Big[2K+A(K,R,T,\delta)\Big](x_j,t_j)\ls 1/j. $$
Letting $j\to \infty$ and noticing that $\psi\gs \epsilon$, we have (recalling $\delta<1/2$, so $1+\ln(1/\delta)\ls 3\ln(1/\delta)$)
$$\sup_{B_{R/2}\times(T/2,T)}F\ls \sup_{B_{5R/8,T}}G\ls 4K+2A(K,R,T,\delta)\ls C_5\cdot\Big(K+\frac{ \ln(1/\delta)}{R^2}+\frac 1 T\Big),$$
where $C_5:=4+6C_4.$ The proof is finished.
\end{proof}

To drop the assumption (\ref{eq3.9}), let us recall an approximation via the Steklov's average.
 \begin{defn}\label{steklov}
 Given $B_R\subset X$ and   $u(x,t)\in L^1(B_{R,T})$, where $B_{R,T}:=B_R\times(0,T)$,  the \emph{Steklov average} of $u$ is defined as follows. For every $\varepsilon\in(0,T)$ and any $h\in  (0,\varepsilon)$,
\begin{equation*}
u_h(x,t):=
\frac{1}{h}\int_0^{h}u(x,t+\tau)d\tau,\quad t\in(0,T-\varepsilon].
\end{equation*}
\end{defn}
By using the standard theory of $L^p$ spaces, it is well know that
 if $u \in  L^p(B_{R,T})$,  then the Steklov average $u_h\to u$ in $L^p(B_{R,T-\varepsilon})$  as $h\to 0$, for  every $\varepsilon\in(0,T).$
 We summarize some necessary  propositions  of the Steklov average as follows, from \cite[Lemma 5.3, Lemma 5.4]{zz-cvpde}.
\begin{lem}[{\cite{zz-cvpde}}]\label{lem3.4}
{\rm (i)}\quad If  $u \in  H^{1}(B_{R,T})\cap L^\infty(B_{R,T})$,  then we have, for every   $\varepsilon\in(0,T) $, that
 $$u_h\in H^{1}(B_{R,T-\varepsilon})\cap L^\infty(B_{R,T-\varepsilon})\quad \ {\rm and}\ \quad
\partial_t u_h\in  H^{1}(B_{R,T-\varepsilon})\cap L^\infty(B_{R,T-\varepsilon})$$
 for  every
$h\in(0,\varepsilon),$ and that $\|u_h\|_{H^{1}(B_{R,T-\varepsilon})}$ is bounded uniformly with respect to $h\in(0,\varepsilon)$.\\
{\rm (ii)}\quad Let $u \in  H^{1}(B_{R,T})\cap L^\infty(B_{R,T})$ be a locally weak solution for the heat equation, and fix any two constants $\varepsilon,h$ such that $\varepsilon\in(0,T) $ and
$h\in(0,\varepsilon)$. Then $u_h$ be a locally weak solution for the heat equation on $B_{R,T-\varepsilon}.$
\end{lem}

Now we can prove Theorem \ref{thm1.1}, via the combination of Lemma \ref{lem3.2} and the approximation by the Steklov's average.
\begin{proof}[Proof of Theorem \ref{thm1.1}]
From \cite[Theorem 2.2]{stu96}, we have $u\in L^\infty_{\rm loc}(B_{R,T})$.
 Since the desired estimate is a local
estimate, without lost of generality, we may assume that $u\in L^\infty(B_{R,T})$.

Fix any $M'>M$, $M'\in\mathbb Q$, and let $v:=u/M'$. Then $v$ is a weak solution of the heat equation on $B_{R,T}$ with $m/M'\ls v\ls M/M'<1$  on $B_{R,T}$.

Given any $\varepsilon>0$ and any $h\in(0,\varepsilon)$, by Lemma \ref{lem3.4}(ii), we know the Steklov averages $v_h$ is a weak solution of the heat equation on $B_{R,T-\varepsilon}$. Now, we can use the Lemma \ref{lem3.2} to $v_h$ and let $h\to0^+$, and conclude that the estimate (\ref{eq3.10}) holds for $v$ on $B_{R/2}\times((T-\varepsilon)/2,T-\varepsilon)$. By the arbitrariness of $\varepsilon$, we get
\begin{equation*}
 |\nabla \ln u|^2(x,t)\ls C_N \cdot \Big( \frac{\ln(M'/m)}{R^2}+\frac{1}{T}+K\Big)\cdot\ln \frac{M'}{ u (x,t)},
\end{equation*}
for almost all $(x,t)\in  B_{R/2}\times{(T/2,T)}$.  The desired result comes from letting $M'\to M$.
\end{proof}
\begin{proof}[Proof of  Theorem \ref{thm1.2}]
Let  $(X,d,\mu)$ be a metric measure space with $RCD^*(K,N)$ and let $H(x,y,t) $ be the heat kernel on $(X,d,\mu)$, whose existence has been proved in \cite{stu95}. From \cite{jlz16} (by taking $\epsilon=1$ in Theorem 1.1 and 1.2 there), there exist positive constants $C_1,C_2$, depending only on $N,K$ and satisfying that if $K=0$ then $C_2=0$, such that
\begin{equation}\label{eq3.19}
\begin{split}
\frac{C_1^{-1}}{\mu\big(B(y,\sqrt t)\big)}&\exp\Big(-\frac{d^2(x,y)}{3t}-C_2\cdot t\Big)\ls H(x,y,t)\\
&\ls
\frac{C_1}{\mu\big(B(y,\sqrt t)\big)}\exp\Big(-\frac{d^2(x,y)}{5t}+C_2\cdot t\Big).
\end{split}
\end{equation}

Fix any $y_0\in X$ and $T,R>0$.  Denote by $B_R:=B(y_0,R)$. Let us consider the function $u(x,t):=H(x,y_0,t)$ on $ B_{R}\times(T/2,T).$
For any two points $(x,t),(x',t')\in B_{R}\times(T/2,T)$, we have, by Bishop-Gromov volume comparison (see  \cite[Proposition 3.6]{eks15}, also \cite[Lemma 2.2]{stu92}) and \eqref{eq3.19}, that
\begin{equation*}\begin{split}
\frac{u(x',t')}{u(x,t)}&\ls C_1^2\cdot \frac{\mu (B_{\sqrt t})}{ \mu (B_{\sqrt{ t'}})}\exp\Big(\frac{d^2(x,y_0)}{3t}+2C_2t\Big)\ls C_1^2\cdot \frac{\mu (B_{\sqrt T})}{ \mu (B_{\sqrt{T/2}})}\exp\Big(\frac{d^2(x,y_0)}{3t}+2C_2t\Big)\\
&\ls C_1^2\cdot C_3 \exp\big((N-1)\sqrt K\cdot\sqrt T\big)\cdot\exp\Big(\frac{d^2(x,y_0)}{3t}+2C_2t\Big),
\end{split}
\end{equation*}
where the constant $C_3$ depends only on $N$.
Then we have, by setting $M=\sup_{B_{R}\times(T/2,T)}u$ and $m=\inf_{B_{R}\times(T/2,T)}u$, that
$$\ln(M/m)\ls  \ln ( C_1^2\cdot C_3) +(N-1)\sqrt {K  T}+  \frac{R^2}{T}+2C_2T\ls C_5+C_4\cdot T+\frac{R^2}{T}$$
and
$$  \ln\frac{M}{u(x,t)}\ls C_5+C_4\cdot T+  \frac{d^2(x,y_0)}{T},$$
where $C_4:= K+2C_2 $ and $C_5:= \ln  (C^2_1\cdot C_3)+(N-1)^2/4$,  and  we have used $d(x,y_0)\ls R$ and $t\in( T/2,T)$.
By using gradient estimate (\ref{eq1.4}) in Theorem \ref{thm1.1}, we have
  \begin{equation*}
  |\nabla \ln H(x,y_0,t)|^2  (x,t)\ls  C_N\Big(\frac{C_5+C_4T}{R^2}+\frac{2}{T}+K \Big)\cdot \Big(C_5+C_4T+ \frac{d^2(x,y_0)}{T}\Big),
\end{equation*}
for almost every $(x,t)\in  B_{R/2}\times(T/2,T)$, where  the constant $C_N$ is given in Theorem \ref{thm1.1}. Therefore, we obtain, by letting $R\to\infty$, that
  \begin{equation*}
  |\nabla \ln H(x,y_0,t)|^2  (x,t)\ls C_N\Big( \frac{2}{T}+K \Big)\cdot \Big(C_5+C_4T+ \frac{d^2(x,y_0)}{T}\Big),
\end{equation*}
for almost every $(x,t)\in  X\times(T/2,T)$. This is the first assertion.
For the case where $K=0$, notice that $C_4=K+2C_2=0$.  This finishes the proof.
\end{proof}
\ \
\begin{proof}[Proof of Corollary \ref{cor1.5}]
Fixed any $t_0\in(0,T)$. We can assume $u$ is bounded and (H\"older) continuous on $B_{3R/4}\times(t_0/4,t_0)$.   Let $v:=u+c$ for some constant $c>0$  large enough.  From Theorem \ref{thm1.1}, we have, for almost all $t\in(t_0/2,t_0)$, that $v(\cdot,t)$ is Lipschitz continuous on $B_{R/2}$ with Lipschitz constant   independent of $t$.  Noticing that $u(\cdot,\cdot)$ is continuous,  by letting $t\to t_0$, we conclude that  $v(\cdot,t_0)$ is Lipschitz continuous on $B_{R/2}$. This    completes the proof.
\end{proof}
 \ \
\begin{proof}[Proof of Corollary \ref{cor1.6}]
Consider $u$ on $B_{R,T}$ for each $R>0$.  Theorem \ref{thm1.1} implies that
$$\sup_{B_R\times (T/2,T)} \frac{ |\nabla f|^2 }{\ln (M/u)}  \ls C_N\cdot\Big(\frac{\ln(M/m)}{R^2}+\frac{1}{T}+K\Big).$$
The desired result comes from letting $R\to\infty$.
 \end{proof}
\ \
\begin{proof}[Proof of Corollary \ref{cor1.8}]
If $N=1$, the metric measure space $(X,d,u)$ satisfies $RCD^*(0,1)$ implies that it satisfies $RCD^*(0,N)$ for any $N>1$. Hence, we can assume that $N>1$.

 Fix $x_0\in X$ and let $B_R:=B(x_0,R)$. For each $R>0$, let $A_R:=\sup_{B_{R}\times(-R^2,0)}|u|$ and
  $$v_R(x,t):=u(x,t-R^2)+2A_R.$$
  We have $A_R\ls v_R\ls 3 A_R$ on $B_{R}\times(0, R^2)$.
 By Theorem \ref{thm1.1}, we get
 $$|\nabla \ln v_R|^2\ls C_N\Big(\frac{\ln 3}{R^2}+\frac{2}{R^2}\Big)\cdot\ln 3,\quad \mu{\rm-a.e.}\ \ {\rm on}\ \ B_{R/2}\times(R^2/2,  R^2).$$
 That is,
  $$|\nabla  u|^2\ls\frac{ 6C_N}{R^2}\cdot \big(u+2A_R\big)^2\ls 6C_N\cdot \frac{ (3A_R)^2}{R^2} ,\quad \mu{\rm-a.e.}\ \ {\rm on}\ \ B_{R/2}\times(-R^2/2, 0).$$
By using the assumption (\ref{eq1.6}) and letting $R\to\infty$, we have $|\nabla  u|=0$ for almost all $(x,t)\in X\times(-\infty,0)$. This finishes the proof.
\end{proof}

\end{document}